\documentclass[a4paper, 11pt]{article}
\usepackage{amsfonts, amssymb, amsmath, amsthm, latexsym,graphicx}

\begin{document}
\baselineskip=22pt

\theoremstyle{plain}
\newtheorem{thm}{Theorem}[section]
\newtheorem{pro}[thm]{Proposition}
\newtheorem{cor}[thm]{Corollary}
\newtheorem{con}[thm]{Conjecture}
\newtheorem{lem}[thm]{Lemma}

\theoremstyle{definition}
\newtheorem{prob}[thm]{Problem}
\newtheorem{rem}[thm]{Remark}
\newtheorem{example}[thm]{Example}

\newcommand{\C}{\textnormal{Con}}
\newcommand{\D}{\textnormal{diag}}
\newcommand{\id}{\textnormal{id}}
\newcommand{\M}{\textnormal{Move}}
\newcommand{\Cy}{\textnormal{Cy}}
\newcommand{\G}{\textnormal{Gy}}

\title{\bf On the eigenvalues of certain Cayley graphs and arrangement graphs}

\author{\large Bai Fan Chen$^{\,\rm 1}$ \quad \quad  Ebrahim Ghorbani$^{\,\rm 2,3}$  \quad \quad Kok Bin Wong$^{\,\rm 1}$\\[.4cm]
{\sl $^{\rm 1}$Institute of Mathematical Sciences, University of Malaya,}\\
{\sl 50603 Kuala Lumpur, Malaysia}\\[0.3cm]
{\sl $^{\rm 2}$Department of Mathematics, K.N. Toosi University of Technology,}\\
{\sl P.O. Box 16315-1618, Tehran, Iran}\\[0.3cm]
{\sl $^{\rm 3}$School of Mathematics, Institute for Research in Fundamental
Sciences (IPM),}\\{\sl P.O. Box
19395-5746, Tehran, Iran }
\\[0.5cm]{
$\mathsf{tufofo1120@gmail.com}$ \quad\quad  $\mathsf{e\_ghorbani@ipm.ir}$ \quad\quad  $\mathsf{kbwong@um.edu.my}$}}

 \maketitle

\begin{abstract}\noindent
In this paper, we show that the eigenvalues of certain classes of Cayley graphs are integers.
 The $(n,k,r)$-arrangement graph  $A(n,k,r)$ is a graph with all the $k$-permutations of an $n$-element set as vertices where two $k$-permutations are adjacent if they differ in exactly $r$ positions.
 We establish a relation between the eigenvalues of the arrangement graphs and the eigenvalues of certain Cayley graphs.
 As a result, the conjecture on integrality of eigenvalues of $A(n,k,1)$ follows.

 \vspace{5mm}
\noindent {\it AMS Mathematics Subject Classification\,(2010):}  05C50, 20C10,  05A05
\end{abstract}

\section{Introduction}

Let $\Gamma$ be a simple graph with vertex set $\nu$. The {\em adjacency matrix}  of   $\Gamma$ is a $\nu\times\nu$ matrix where its rows and
columns  indexed by  the vertex set of $\Gamma$  and its $(u, v)$-entry  is $1$ if the vertices $u$ and
$v$ are adjacent and $0$ otherwise. By {\em eigenvalues} of $\Gamma$ we mean the eigenvalues of its adjacency matrix.
A graph  is said to be \emph{integral} if all its eigenvalues are integers. All graphs considered are finite (multi-)graphs without self-loops.

\subsection{Cayley graphs}

Let $G$ be a finite group and $S$ be an inverse closed subset of $G$. The \emph{Cayley graph} $\Gamma(G,S)$ is the graph which has the elements of $G$ as its vertices and two vertices $u, v \in G$ are joined by an edge if and only if $v=su$ for some $s\in S$.

 A Cayley graph $\Gamma(G,S)$ is said to be {\em normal} if $S$ is closed under conjugation. It is well known that the eigenvalues of a normal Cayley graph $\Gamma(G,S)$ can be expressed in terms of the irreducible characters of $G$.

\begin{thm}[\cite{Babai, DS, Lub, Ram}]\label{cayley}
The eigenvalues of a normal Cayley graph $\Gamma(G,S)$ are given by
\begin{eqnarray}
\eta_{\chi} & = & \frac{1}{\chi(1)} \sum_{s \in S} \chi(s),\notag
\end{eqnarray}
where $\chi$ ranges over all the irreducible characters of $G$. Moreover, the multiplicity of $\eta_{\chi}$ is $\chi(1)^{2}$.
\end{thm}

Let $\mathcal{S}_{n}$ be the symmetric group on $[n]=\{1, \ldots, n\}$ and $S\subseteq \mathcal S_n$ be closed under conjugation. Since central
characters are algebraic integers (\cite[Theorem 3.7 on p. 36]{Isaacs}) and that the characters of the symmetric group are integers (\cite[2.12 on p. 31]{Isaacs} or \cite[Corollary 2 on p. 103]{Serre}), by Theorem \ref{cayley}, the eigenvalues of $\Gamma(\mathcal S_n,S)$ are integers.

\begin{cor}\label{cor_cayley_integer}
The eigenvalues of a normal Cayley graph $\Gamma(\mathcal S_n,S)$ are integers.
\end{cor}

In general, if $S$ is not closed under conjugation, then the eigenvalues of $\Gamma(\mathcal S_n,S)$ may not be integers \cite{Friedman} (see also \cite{Abdo, Kra, Renteln2} for related results on the eigenvalues of certain Cayley graphs).

\begin{prob}\label{problem_1} Find conditions on $S$, so that the eigenvalues of $\Gamma(\mathcal S_n,S)$  are integers.
\end{prob}

Let $2\leq r\leq n$ and $\Cy(r)$ be the set of all $r$ cycles in $\mathcal S_n$ which do not fix 1, i.e.
\begin{equation}
\Cy(r)=\{ \alpha\in\mathcal S_n \mid \alpha(1)\neq 1\ \textnormal{and $\alpha$ is an $r$-cycle}\}.\notag
\end{equation}
For instance, $\Cy(2)=\{ (1\ 2), (1\ 3),\dots ,(1\ n)\}$. It was conjectured by  Abdollahi and Vatandoost \cite{Abdo} that  the eigenvalues of $\Gamma(\mathcal S_n,\Cy(2))$  are integers, and contains all integers in the range from $-(n-1)$ to $n-1$ (with the sole exception that when $n=2$ or $3$, zero is not an eigenvalue of $\Gamma(\mathcal S_n,\Cy(2))$.  The second part of the conjecture was proved by Krakovski and Mohar \cite{Kra}. In fact, they showed that for $n\geq 2$ and each integer $1\leq l\leq n-1$, $\pm (n-l)$ are eigenvalues of $\Gamma(\mathcal S_n,\Cy(2))$ with multiplicity at least $\binom{n-2}{l-1}$. Furthermore, if $n\geq 4$, then 0 is an eigenvalue of $\Gamma(\mathcal S_n,\Cy(2))$ with multiplicity at least $\binom{n-1}{2}$. Later, Chapuy and F\'{e}ray \cite {cf} pointed out that the conjecture could be proved by using Jucys-Murphy elements.
In this paper, we  generalize this to the following:

\begin{thm}\label{thm_main2} The eigenvalues of ~$\Gamma(\mathcal S_n,\Cy(r))$  are integers.
\end{thm}

In fact, Theorem \ref{thm_main2} follows from Theorem \ref{thm_main1} which states that for certain  subsets $S$ of $\mathcal S_n$, the eigenvalues of $\Gamma(\mathcal S_n,S)$ are integers.

\subsection{Arrangement graphs}

For $k\le n$, a {\em $k$-permutation} of $[n]$ is an injective function from $[k]$ to $[n]$.
So any $k$-permutation $\pi$ can be represented by a vector $(i_1,\ldots,i_k)$ where $\pi(j)=i_j$ for $j=1,\ldots,k$. Let $1\leq r\leq k\leq n$.
The {\em $(n,k,r)$-arrangement graph}   $A(n,k,r)$ has all the $k$-permutations of $[n]$ as vertices and two $k$-permutations are adjacent if they differ in exactly $r$ positions. Formally, the vertex set $V(n,k)$ and edge set $E(n,k,r)$ of $A(n,k,r)$ are
\begin{align}
V(n,k) & =\big\{(p_1,p_2,\dots ,p_k) \mid p_i\in [n]\ \textnormal{and}\ p_i\neq p_j\ \textnormal{for}\ i\neq j\big\},\notag\\
E(n,k,r) &=\big\{ \{(p_1,p_2,\dots ,p_k),(q_1,q_2,\dots ,q_k)\} \mid p_i\neq q_i \ \textnormal {for $i\in R$ and}\notag\\
&\hskip 1cm p_j= q_j \ \textnormal{for all $j\in [k]\setminus R$ for some $R\subseteq [k]$ with $\vert R\vert=r$}\big\}.\notag
\end{align}
Note that $\vert V(n,k)\vert=n!/(n-k)!$ and $A(n,k,r)$ is a regular graph (part (c) of Theorem  \ref{thm_main3}). In particular, $A(n,k,1)$ is a $k(n-k)$-regular graph. We note here that $A(n,k,1)$ is called partial permutation graph in \cite{Kra}.

The family of the arrangement graphs $A(n,k,1)$ was first introduced in \cite{dt} as an interconnection network model for parallel computation.
In the interconnection network model, each processor has its own memory unit and
communicates with the other processors through a topological network, i.e. a graph.
For this purpose, the arrangement graphs possess many nice properties such as having small diameter, a hierarchical structure, vertex and edge symmetry, simple shortest path routing, high connectivity, etc. Many properties of the arrangement graphs $A(n,k,1)$ have been studied in \cite{cll,cly,cgq,cqs,cc,dt,ttt,zx}.

Let us look at the eigenvalues of $A(n,k,1)$. Since $A(n,k,1)$ is a $k(n-k)$-regular graph, the largest eigenvalue is $k(n-k)$. When $n=k$, the edge set is an empty set. So,  $A(k,k,1)$ has one eigenvalue only, which is 0. When $k=1$, $A(n,1,1)$ is the complete graph with $n$ vertices. Therefore the eigenvalues of $A(n,1,1)$ are $(n-1)$ with multiplicity 1, and $-1$ with multiplicity $n-1$.
The eigenvalues of the arrangement graphs $A(n,k,1)$ were first studied in \cite{Chen} by using a method developed by Godsil and McKay  \cite{gm}.

In this paper, we will study the eigenvalues of $A(n,k,r)$. We will give a relation between the eigenvalues of $A(n,k,r)$ and the eigenvalues of certain Cayley graphs (Theorem \ref{thm_main3}). We then apply Theorem \ref{thm_main1} to prove following theorem which was conjectured in \cite{Chen}.

\begin{thm}\label{thm_main4} The eigenvalues of $A(n,k,1)$ are integers.
\end{thm}

\section{Preliminaries}

 This section contains preliminary materials which will be used to prove the results of the paper.

\subsection{Adjacency matrix of Cayley graphs}


Let $G$ be a finite group. Note that the ring $\mathbb C[G]$ can be considered as a vector space over $\mathbb C$ with the {\em standard basis} $\{1=g_1,g_2,\dots, g_m\}=G$. Let $S$ be  an inverse closed subset of $G$. Then
\begin{equation}
\left(\sum_{s\in S} s\right) g_i=\sum_{j=1}^m b_{ij}g_j,\notag
\end{equation}
where $b_{ij}\in\{ 0,1\}$. Let $B=[b_{ij}]$. Then we may write
\begin{equation}
\left(\sum_{s\in S} s\right) \mathbf g =B\mathbf g,\notag
\end{equation}
where $\mathbf g=(g_1,g_2,\dots, g_m)^\top$. Note that $B$ is the adjacency matrix of $\Gamma (G,S)$. 

Let $H$ be a subgroup of $G$ and $G=\bigcup_{i=1}^l Ha_i$ be the disjoint union of all the right cosets of $H$. We may assume that $a_1=1$. Let $H=\{1=h_1,\dots, h_e\}$. Then
$$G=\{1=h_1,\dots, h_e, h_1a_2,\dots, h_ea_2,\dots, h_1a_l,\dots, h_ea_l\}$$
 is the standard basis for $\mathbb C[G]$ and $\mathbb C[G]$ can be decomposed into sum of $\mathbb C[H]$-modules, i.e.
\begin{equation}
\mathbb C[G]=\bigoplus_{i=1}^l \mathbb C[Ha_i].\notag
\end{equation}
Note that $\mathbb C[Ha_i]$ and $\mathbb C[H]$ are isomorphic $\mathbb C[H]$-modules. The following lemma is obvious.

\begin{lem}\label{lemma_breakup}
Suppose that $S$ is a subset of $H$. If $B_H$ is the adjacency matrix of $\Gamma (H,S)$, then the adjacency matrix $B_G$ of $\Gamma (G,S)$ is
\begin{equation}
B_G=\begin{pmatrix}
B_H &  & \\
 & \ddots &\\
&  & B_H
\end{pmatrix}.\notag
\end{equation}
In fact, $\Gamma (G,S)$ is the disjoint union of $l$ copies of $\Gamma (H,S)$.
\end{lem}

\subsection{Symmetric groups}

 Let $\alpha,\beta\in \mathcal S_n$. Then  $\alpha\beta=\beta\circ \alpha$, i.e. $\alpha\beta(i)=\alpha(\beta(i))$ for all $i\in [n]$. For example, if $\alpha=(1\ \ 2 \ \ 3)$ and $\beta=(2\ \ 3)$, then
\begin{equation}
\alpha\beta=(1\ \ 2 \ \ 3)(2\ \ 3)=(1 \ \ 3).\notag
\end{equation}

\begin{rem} Note that some authors define $\alpha\beta$ as $\alpha\beta(i)=\beta(\alpha(i))$. If one follows this definition, then $(1\ \ 2 \ \ 3)(2\ \ 3)=(1 \ \ 2)$.
\end{rem}

Let $(c_1\ \ c_2\ \cdots \ c_l)$ be an $l$-cycle in $\mathcal S_n$. Then
\begin{equation}\label{eq_conjugate}
\alpha^{-1}\ (c_1\ \ c_2\ \cdots \ c_l)\ \alpha=(\alpha(c_1)\ \ \alpha(c_2)\  \cdots \ \alpha(c_l)),
\end{equation}
which is also an $l$-cycle.

Note that every element in $\mathcal S_n$ can be decomposed into product of disjoint cycles and the decomposition is unique. Let $a_1,\dots, a_n$ be non-negative integers such that $\sum_{i=1}^n ia_i=n$.
An element $\alpha\in \mathcal S_n$ is said to be of type  $1^{a_1}2^{a_2}\ldots n^{a_n}$ if $\alpha$ has exactly $a_i$ number of  $i$-cycles in its decomposition. Let $\mathcal K(a_1,\dots, a_n)$ be the set of all elements of $\mathcal S_n$ of type  $1^{a_1}2^{a_2}\ldots n^{a_n}$. By \eqref{eq_conjugate}, one can deduce that $\mathcal K(a_1,\dots, a_n)=\{ g^{-1}ug \mid  g\in \mathcal S_n\}$ where $u$ is an element of type $1^{a_1}2^{a_2}\ldots n^{a_n}$. Furthermore,
$$\vert \mathcal K(a_1,\dots, a_n)\vert=\frac{n!}{\prod_{i=1}^n a_i!i^{a_i}}.$$

\section{Integrality}

In this section we establish the integrality of certain classes of Cayley graphs. As a corollary, Theorem~\ref{thm_main2} follows.

For each $u\in G$, let $\C_G(u)$ be the conjugacy class of $u$ in $G$, i.e.
\begin{equation}
\C_G(u)=\{ g^{-1}ug\mid  g\in G\}.\notag
\end{equation}


A pair of subsets $S_1$ and $S_2$ of $G$ are said to be {\em commutative} if
\begin{equation}
\left(\sum_{s\in S_1} s\right)\left(\sum_{s\in S_2} s\right)=\left(\sum_{s\in S_2} s\right)\left(\sum_{s\in S_1} s\right),\notag
\end{equation}
in the ring $\mathbb C[G]$.

We shall use the following well-known fact (see, e.g. \cite[p. 92]{hj}).

\begin{lem}\label{lm_commute_eigen} Let $B$ and $C$ be two real $m\times m$ symmetric matrices and $r,s\in \mathbb R$. If $BC=CB$, then every eigenvalue of $rB+sC$ is of the form $r\lambda+s\gamma$ where $\lambda$ and $\gamma$ are eigenvalues of $B$ and $C$, respectively.
\end{lem}

\begin{lem}\label{lm_integer_eigen} Let $\Gamma(G,S_1)$ and $\Gamma(G,S_2)$ be two integral Cayley graphs. Suppose that $S_1$ and $S_2$ are commutative. Then the following assertions hold.
\begin{itemize}
\item[\textnormal{(a)}] If $S_1\cap S_2=\emptyset$, then $\Gamma(G,S_1\cup S_2)$ is integral.
\item[\textnormal{(b)}] If $S_2\subseteq S_1$, then $\Gamma(G,S_1\setminus S_2)$ is integral.
\end{itemize}
\end{lem}

\begin{proof} Let $A_1$ and $A_2$ be the adjacency of the Cayley graph corresponding to $S_1$ and $S_2$, respectively. Now, commutativity of $S_1$ and $S_2$ implies that $A_1A_2=A_2A_1$ (see Section 2.1).  Since all the eigenvalues of $A_1$ and $A_2$ are integers, by Lemma \ref{lm_commute_eigen}, the eigenvalues of $A_1+A_2$ and $A_1-A_2$ are integers. The lemma follows by noting that $A_1+A_2$ and $A_1-A_2$ are  the adjacency matrices of $\Gamma(G,S_1\cup S_2)$ (when $S_1\cap S_2=\emptyset$) and
$\Gamma(G,S_1\setminus S_2)$ (when $S_2\subseteq S_1$), respectively.
\end{proof}

Let $\alpha\in\mathcal S_n$. We denote the set of elements in $[n]$ moved by $\alpha$ by $\M(\alpha)$, i.e.
\begin{equation}
\M(\alpha)=\{ i\in [n] \mid  \alpha(i)\neq i\}.\notag
\end{equation}
Let $T\subseteq [n]$ and
\begin{equation}
\mathcal S_n(T)=\{ \alpha\in \mathcal S_n \mid  \M(\alpha)\subseteq [T]\}.\notag
\end{equation}
In fact, $\mathcal S_n(T)$ are all those permutations that fix the complement of $T$. Let $\id$ denote the identity element in $\mathcal S_n$. Then $\M(\id)=\emptyset$ and $\id\in \mathcal S_n(T)$. So, $\mathcal S_n(T)\neq \emptyset$. In fact, $\mathcal S_n(T)$ is a subgroup of $\mathcal S_n$ and is isomorphic to $\mathcal S_{\vert T\vert}$.

For a partition $T_1,\dots ,T_l$ of $[n]$, i.e. $[n]=\bigcup_{i=1}^l T_i$ and $T_i\cap T_j=\emptyset$ for $i\neq j$, we set
\begin{equation}
\mathcal T(T_1,\dots, T_l)=\mathcal S_n\setminus \left( \bigcup_{i=1}^l \mathcal S_n(T_i)  \right).\notag
\end{equation}

A subset $S$ of $\mathcal S_n$ is said to be \emph{nicely separated} if there exists a partition $T_1,\dots ,T_l$ ($l\geq 2$) of $[n]$ such that
\begin{equation}
S=\bigcup_{s\in S} \C_{\mathcal S_n}(s)\cap \mathcal T(T_1,\dots, T_l).\notag
\end{equation}

\begin{lem}\label{lm_symmetric} If $S$ is nicely separated, then  $s\in S$ implies that $s^{-1}\in S$.
\end{lem}

\begin{proof} Let $s\in S$. Then $s^{-1}\in \C_{\mathcal S_n}(s)$ for $s$ and $s^{-1}$ must be of the same type (see Section 2.3).
Since $s\in \mathcal T(T_1,\dots, T_l)$, there exist $a\in T_i$ and $b\in T_j$ such that $a,b\in \M(s)$ for some $i,j$ with $i\neq j$. This implies that $a,b\in \M(s^{-1})=\M(s)$ and $s^{-1}\in \mathcal T(T_1,\dots, T_l)$. Hence $s^{-1}\in S$.
\end{proof}

Now, $\id\notin S$ for $\id\notin \mathcal T(T_1,\dots, T_l)$. Hence $\Gamma(\mathcal S_n, S)$ is a Cayley graph.

\begin{thm}\label{thm_main1} If $S$ is nicely separated, then $\Gamma(\mathcal S_n, S)$ is integral.
\end{thm}

\begin{proof} Note that
\begin{equation}\label{eq_defining}
S=\bigcup_{s\in S} \C_{\mathcal S_n}(s)\setminus \left( \bigcup_{1\leq i\leq l,\,s\in S} \mathcal S_n(T_i)\cap \C_{\mathcal S_n}(s)  \right).
\end{equation}
If $\mathcal S_n(T_i)\cap \C_{\mathcal S_n}(s)=\emptyset$, we shall ignore this term in \eqref{eq_defining}. Suppose that $\mathcal S_n(T_i)\cap \C_{\mathcal S_n}(s)\neq \emptyset$. Let $\alpha\in \mathcal S_n(T_i)\cap \C_{\mathcal S_n}(s)$. Then $\C_{\mathcal S_n(T_i)}(\alpha)\subseteq \mathcal S_n(T_i)\cap \C_{\mathcal S_n}(s)$. Therefore $\mathcal S_n(T_i)\cap \C_{\mathcal S_n}(s)$ is closed under conjugation with elements in $\mathcal S_n(T_i)$. Let
\begin{equation}
S_i=\bigcup_{s\in S} \mathcal S_n(T_i)\cap \C_{\mathcal S_n}(s).\notag
\end{equation}
Using similar argument as in  Lemma \ref{lm_symmetric},  we see that $s\in S_i$ implies that $s^{-1}\in S_i$. Furthermore, $\id\notin S_i$. Hence $\Gamma(\mathcal S_n, S_i)$ is a Cayley graph.
By Corollary \ref{cor_cayley_integer}, $\Gamma (\mathcal S_n(T_i),S_i)$ is integral.
  $\Gamma(\mathcal S_n, S_i)$ is a disjoint union of $[\mathcal S_n:\mathcal S_n(T_i)]$ copies of $\Gamma (\mathcal S_n(T_i),S_i)$ (Lemma \ref{lemma_breakup})  implying that  $\Gamma(\mathcal S_n, S_i)$ is integral.

Let $S_0=\bigcup_{s\in S} \C_{\mathcal S_n}(s)$. By Corollary~\ref{cor_cayley_integer}, $\Gamma(\mathcal S_n, S_0)$ is integral. Now,
\begin{equation}
S=S_0\setminus \left(\bigcup_{1\leq i\leq l} S_i  \right).\notag
\end{equation}
Note that $S_i$ and $S_j$ are commutative and  $S_i\cap S_j=\emptyset$ for $1\leq i<j\leq l$. By applying part (a) of Lemma~\ref{lm_integer_eigen} repeatedly,
we conclude that $\Gamma\left(\mathcal S_n, \bigcup_{i=1}^l S_i\right)$ is integral.
Since $S_0$ is a union of conjugacy classes in $\mathcal S_n$, $S_0$ and $\bigcup_{i=1}^l S_i$ is commutative. Hence, by part (b) of Lemma \ref{lm_integer_eigen}, $\Gamma(\mathcal S_n, S)$ is integral.
\end{proof}

\begin{proof}[Proof of Theorem \ref{thm_main2}] It follows from Theorem \ref{thm_main1} by noting that $\Cy(r)$ is nicely separated with $T_1=\{1\}$ and $T_2=\{2,3,\dots, n\}$.
\end{proof}

Recall that $\mathcal K(a_1,\dots, a_n)=\{ g^{-1}ug \mid  g\in \mathcal S_n\}$ where $u$ is an element of $\mathcal S_n$ which has exactly $a_i$ number of  $i$-cycles in its cycle decomposition.

\begin{cor}\label{cor_consequence} Let $T_1,\dots ,T_l$ be a partition of $[n]$ and $S=\mathcal K(a_1,\dots, a_n)\setminus \bigcup_{i=1}^l \mathcal S_n(T_i)$. Then   $\Gamma(\mathcal S_n, S)$ is integral.
\end{cor}

\begin{example}\label{ex_1} $\Gamma(\mathcal S_n, S)$ is integral for the following:
\begin{itemize}
\item[\textnormal{(a)}] $S=\{ (i,j) \mid i\in [k], j\in [n]\setminus [k]\}$;
\item[\textnormal{(b)}] $S=\{ (i_1,i_2,i_3) \mid \{i_1,i_2,i_3\}\cap [k]\neq \emptyset\ \textnormal{and}\ \{i_1,i_2,i_3\}\cap ([n]\setminus [k])\neq \emptyset\}$;
\item[\textnormal{(c)}] $S=\{ (i_1,i_2)(i_3,i_4) \mid  \{i_1,i_2,i_3,i_4\}\cap [k]\neq\emptyset\ \textnormal{and}\  \{i_1,i_2,i_3\}\cap ([n]\setminus [k])\neq\emptyset\}$.
\end{itemize}
\end{example}

\section{Eigenvalues of the arrangement graphs}\label{Eigenvalues Ankr}

In this section we establish a connection between the eigenvalues of the arrangement graphs $A(n,k,r)$ and the eigenvalues of certain Cayley graphs.
More precisely, we prove that the adjacency matrix of $A(n,k,r)$ (modulo an integer factor) is a quotient matrix of an equitable partition of certain Cayley graphs.

An {\em equitable partition} of a graph $\Gamma$ is a partition $\Pi=(V_1,\ldots,V_m)$ of the vertex set
such that each vertex in $V_i$ has the same number $q_{ij}$ of neighbors in  $V_j$
for any $i,j$ (and possibly $i=j$). The {\em quotient matrix} of $\Pi$ is the $m\times m$ matrix $Q=(q_{ij})$.
It is well-known that every eigenvalue of the quotient matrix $Q$ is an eigenvalue of $\Gamma$ (see \cite[p. 24]{bh}).

Let $n\geq k\geq r\geq 1$ and
\begin{equation}
M(r)=\{ \sigma\in\mathcal S_n \mid \vert \M(\sigma)\cap [k]\vert=r\}.\notag
\end{equation}
Basically, $M(r)$ is the set of all elements in $\mathcal S_n$ that move exactly $r$ elements in $[k]$. Note that $M(r)$ is inverse closed. Hence $\Gamma(\mathcal S_n, M(r))$ is a Cayley graph.

Let $T=[n]\setminus [k]$ and $\mathcal S_n=\bigcup_{i=1}^{l} \mathcal S_n(T)\alpha_i$ be the disjoint union of all the right cosets of $\mathcal S_n(T)$, where $\alpha_i\in\mathcal S_n$ are the right cosets representative. Note that $l=n!/(n-k)!$.

\begin{lem}\label{lm_equi}\
\begin{itemize}
\item[\textnormal{(a)}] For $1\leq i\leq l$, $\beta\alpha_i$ is not adjacent to $\delta\alpha_i$ in $\Gamma(\mathcal S_n, M(r))$ for all $\beta, \delta\in\mathcal S_n(T)$.
\item[\textnormal{(b)}] For $1\leq i,j\leq l$, $i\neq j$ and $\beta_0, \delta_0\in\mathcal S_n(T)$,  if $\beta_0\alpha_i$ is adjacent to $\delta_0\alpha_j$ in $\Gamma(\mathcal S_n, M(r))$, then  $\beta_0\alpha_i$ is adjacent to $\delta\alpha_j$ in $\Gamma(\mathcal S_n, M(r))$ for all $\delta\in\mathcal S_n(T)$.
\end{itemize}

\end{lem}

\begin{proof} (a) Suppose $\beta\alpha_i$ is adjacent to $\delta\alpha_i$ for some $\beta,\delta\in\mathcal S_n(T)$. Then there is a $\gamma\in M(r)$ such that $\beta\alpha_i=\gamma\delta\alpha_i$. This implies that $\gamma=\beta\delta^{-1}\in \mathcal S_n(T)$, a contradiction.

\vskip 0.5cm
\noindent
(b) Since $\beta_0\alpha_i$ is adjacent to $\delta_0\alpha_j$,  there is a $\gamma\in M(r)$ such that $\beta_0\alpha_i=\gamma\delta_0\alpha_j$. Let $\delta \in \mathcal S_n(T)$. Then $\beta_0\alpha_i=(\gamma\delta_0\delta^{-1})\delta\alpha_j$. Since $T\cap [k]=\emptyset$, $\gamma\delta_0\delta^{-1}\in M(r)$. Hence $\beta_0\alpha_i$ is adjacent to $\delta\alpha_j$.
\end{proof}

\begin{thm}\label{thm_main3} \
\begin{itemize}
\item[\textnormal{(a)}] $(\mathcal S_n(T)\alpha_1,\ldots,\mathcal S_n(T)\alpha_l)$ is an  equitable partition of $\Gamma(\mathcal S_n, M(r))$.
\item[\textnormal{(b)}] Let $Q$ be the  quotient matrix of $(\mathcal S_n(T)\alpha_1,\ldots,\mathcal S_n(T)\alpha_l)$ and $A_{n,k,r}$ be the adjacency matrix of the arrangement graph $A(n,k,r)$. Then
\begin{equation}
Q=(n-k)!A_{n,k,r}.\notag
\end{equation}
\item[\textnormal{(c)}] $A(n,k,r)$ is a $\frac{\vert M(r)\vert}{(n-k)!}$-regular graph.
\item[\textnormal{(d)}] If $\lambda$ is an eigenvalue of $A(n,k,r)$, then $(n-k)!\lambda$  is an eigenvalue of $\Gamma(\mathcal S_n, M(r))$.
\end{itemize}
\end{thm}

\begin{proof} By Lemma \ref{lm_equi}, a vertex in $\mathcal S_n(T)\alpha_i$ is not adjacent to any vertex in $\mathcal S_n(T)\alpha_i$, furthermore, if a vertex in $\mathcal S_n(T)\alpha_i$ is adjacent to a vertex in $\mathcal S_n(T)\alpha_j$ and $i\neq j$, then it is adjacent to all the vertices in $\mathcal S_n(T)\alpha_j$.
This proves part (a) and that  every entry in $Q$ is either 0 or $(n-k)!$.

Note that we can represent the right coset $\mathcal S_n(T)\alpha_i$ by the vector $(\alpha_i(1),\alpha_i(2),\dots, \alpha_i(k))$. This is because if $(\alpha_i(1),\alpha_i(2),\dots, \alpha_i(k))=(\alpha_j(1),\alpha_j(2),\dots, \alpha_j(k))$ for $i\neq j$, then $\alpha_i\alpha_j^{-1}\in \mathcal S_n(T)$, which is impossible as $\mathcal S_n(T)\alpha_i$ and $\mathcal S_n(T)\alpha_j$ are distinct right cosets. Let $Q=[q_{ij}]$. Then $q_{ij}=(n-k)!$ exactly when the two vectors $(\alpha_i(1),\alpha_i(2),\dots, \alpha_i(k))$ and $(\alpha_j(1),\alpha_j(2),\dots, \alpha_j(k))$ differ in exactly $r$ positions. This proves part (b).

The sum $\sum_{j=1}^l q_{ij}$ is equal to the degree of a vertex in $\mathcal S_n(T)\alpha_i$ in the graph $\Gamma(\mathcal S_n, M(r))$. So, $\sum_{j=1}^l q_{ij}=\vert M(r)\vert$. On the other hand, $\sum_{j=1}^l q_{ij}$ is equal to $(n-k)!$ times the degree of the vertex $(\alpha_i(1),\alpha_i(2),\dots, \alpha_i(k))$ in the arrangement graph $A(n,k,r)$. Hence  the degree of the vertex $(\alpha_i(1),\alpha_i(2),\dots, \alpha_i(k))$ in $A(n,k,r)$ is $\frac{\vert M(r)\vert}{(n-k)!}$ and part (c) follows.

Part (d) follows by noting that $(n-k)!\lambda$ is an eigenvalue of $Q$ and every eigenvalue of $Q$ is an eigenvalue of $\Gamma(\mathcal S_n, M(r))$.
\end{proof}

\section{Proof of Theorem \ref{thm_main4}}

In this section we give a proof for Theorem~\ref{thm_main4} which is based on the method given in Section~\ref{Eigenvalues Ankr}.

Let $A(n,k)=A(n,k,1)$. Note that $A(k,k)$  is the empty graph with $k$ vertices and $A(n,1)$ is the complete graph with $n$ vertices. So, it is sufficient to prove that the eigenvalues of $A(n,k)$ are integers for $n>k\geq 2$.

Let  $S=\{ (i,j) \mid i\in [k], j\in [n]\setminus [k]\}$.
Let $T=[n]\setminus [k]$ and $\mathcal S_n=\bigcup_{i=1}^{l} \mathcal S_n(T)\alpha_i$ be the disjoint union of all the right cosets of $\mathcal S_n(T)$. Note that $l=n!/(n-k)!$.

\begin{lem}\label{lm_equi2}\
\begin{itemize}
\item[\textnormal{(a)}] For $1\leq i\leq l$, $\beta\alpha_i$ is not adjacent to $\delta\alpha_i$ in $\Gamma(\mathcal S_n,S)$ for all $\beta, \delta\in\mathcal S_n(T)$.
\item[\textnormal{(b)}] For $1\leq i,j\leq l$, $i\neq j$ and $\beta_0, \delta_0\in\mathcal S_n(T)$,  if $\beta_0\alpha_i$ is adjacent to $\delta_0\alpha_j$ in $\Gamma(\mathcal S_n, S)$, then for any $\beta\in\mathcal S_n(T)$, there is a unique $\delta\in\mathcal S_n(T)$ such that $\beta\alpha_i$ is adjacent to $\delta\alpha_j$ in $\Gamma(\mathcal S_n, S)$.
\end{itemize}

\end{lem}

\begin{proof} (a) Suppose $\beta\alpha_i$ is adjacent to $\delta\alpha_i$ for some $\beta,\delta\in\mathcal S_n(T)$. Then $\beta\alpha_i=(a\ \ b)\delta\alpha_i$ with $a\in [k]$ and $b\in [n]\setminus [k]$. This implies that $(a\ \ b)=\beta\delta^{-1}\in \mathcal S_n(T)$, a contradiction.

\vskip 0.5cm
\noindent
(b) Since $\beta_0\alpha_i$ is adjacent to $\delta_0\alpha_j$,  $\beta_0\alpha_i=(a\ \ b)\delta_0\alpha_j$ with $a\in [k]$ and $b\in [n]\setminus [k]$. Let $\beta \in \mathcal S_n(T)$. Then $\beta\alpha_i=\beta(\beta_0^{-1}\beta_0)\alpha_i=(\beta\beta_0^{-1})(a\ \ b)\delta_0\alpha_j=((\beta\beta_0^{-1})(a\ \ b)(\beta\beta_0^{-1})^{-1})(\beta\beta_0^{-1}\delta_0)\alpha_j$. Since $T\cap [k]=\emptyset$, $(\beta\beta_0^{-1})(a\ \ b)(\beta\beta_0^{-1})^{-1}=(a \ \ c)\in S$ where $c=(\beta\beta_0^{-1})^{-1}(b)\in [n]\setminus [k]$
(see \eqref{eq_conjugate}). Hence $\beta\alpha_i$ is adjacent to $\delta\alpha_j$ where $\delta=\beta\beta_0^{-1}\delta_0$.

Suppose $\beta\alpha_i$ is adjacent to $\delta'\alpha_j$. Then $\beta\alpha_i=(a'\ \ b')\delta'\alpha_j$ with $a'\in [k]$ and $b'\in [n]\setminus [k]$. On the other hand, $\beta\alpha_i=(a\ \ c)\delta\alpha_j$. Therefore $(a\ \ c)(a'\ \ b')=\delta{\delta'}^{-1}\in\mathcal S_n(T)$ and this is only possible when $a=a'$ and $c=b'$. Hence $\delta=\delta'$. This completes the proof of part (b).
\end{proof}

By Lemma \ref{lm_equi2}, $(\mathcal S_n(T)\alpha_1,\ldots,\mathcal S_n(T)\alpha_l)$ is an  equitable partition of $\Gamma(\mathcal S_n, S)$. Let $Q=[q_{ij}]$ be the  quotient matrix of $(\mathcal S_n(T)\alpha_1,\ldots,\mathcal S_n(T)\alpha_l)$. Then  every entry in $Q$ is either 0 or $1$. If we represent the right coset $\mathcal S_n(T)\alpha_i$ by the vector $(\alpha_i(1),\alpha_i(2),\dots, \alpha_i(k))$, then  $q_{ij}=1$ if and only if the two vectors $(\alpha_i(1),\alpha_i(2),\dots, \alpha_i(k))$ and $(\alpha_j(1),\alpha_j(2),\dots, \alpha_j(k))$ differ in exactly $1$ position. Thus,
\begin{equation}
Q=A_{n,k},\notag
\end{equation}
where $A_{n,k}$ is the adjacency matrix of the arrangement graph $A(n,k)$. This implies that the eigenvalues of $A(n,k)$ are eigenvalues of $\Gamma(\mathcal S_n, S)$.
Now $\Gamma(\mathcal S_n, S)$ is integral (see Example \ref{ex_1}). Hence all the eigenvalues of $A(n,k)$ are integers and Theorem~\ref{thm_main4} follows.

\section*{Acknowledgments} 
We would like to thank the anonymous referee for the comments that helped us make several improvements to this
paper. 

The first and the third authors are supported by the Advanced Fundamental Research Cluster, University of Malaya
(UMRG RG238/12AFR). The second author is supported by a grant from IPM (No. 92050114).

\end{document}